\documentclass[12pt,reqno]{amsart}

\usepackage{latexsym}
\usepackage{amssymb}

\newcommand{\lla}{\left\langle}
\newcommand{\rra}{\right\rangle}

\newcommand{\s}{{\mathbf s}}

\renewcommand{\epsilon}{\varepsilon}
\newcommand{\dist}{{\operatorname{dist}}}

\newcommand{\sm}{\smallsetminus}
\newcommand{\szego}{Szeg\H{o} }

\newcommand{\inv}{^{-1}}
\newcommand{\kahler}{K\"ahler }

\newcommand{\wh}{\widehat}
\newcommand{\PP}{{\mathbb P}}

\newcommand{\R}{{\mathbb R}}
\newcommand{\C}{{\mathbb C}}

\newcommand{\Z}{{\mathbb Z}}

\newcommand{\CP}{\C\PP}
\renewcommand{\d}{\partial}
\newcommand{\dbar}{\bar\partial}
\newcommand{\ddbar}{\partial\dbar}

\newcommand{\vol}{{\operatorname{Vol}}}

\renewcommand{\phi}{\varphi}

\newcommand{\ccal}{\mathcal{C}}

\newcommand{\hcal}{\mathcal{H}}

\newcommand{\ocal}{\mathcal{O}}

\newcommand{\al}{\alpha}
\newcommand{\be}{\beta}
\newcommand{\ga}{\gamma}
\newcommand{\Ga}{\Gamma}

\newcommand{\la}{\lambda}
\newcommand{\ep}{\varepsilon}
\newcommand{\de}{\delta}
\newcommand{\De}{\Delta}
\newcommand{\om}{\omega}

\newtheorem{theo}{{\sc Theorem}}[section]

\newtheorem{cor}[theo]{{\sc Corollary}}

\newtheorem{lem}[theo]{{\sc Lemma}}
\newtheorem{prop}[theo]{{\sc Proposition}}

\newenvironment{rem}{\medskip\noindent{\it Remark:\/} }{\medskip}

\title[Uniformly bounded orthonormal sections]
{Uniformly bounded orthonormal sections of positive line bundles on complex manifolds}

\author{Bernard Shiffman}

\address{Department of Mathematics, Johns Hopkins University, Baltimore, MD
21218, USA} \email{shiffman@math.jhu.edu}

\thanks{Research  partially supported by NSF grant DMS-1201372.}

\dedicatory{To  Duong H. Phong  on the occasion of his 60th birthday}

\date{April 5, 2014}

\begin{document}

\begin{abstract}
We show the existence of uniformly bounded sequences of increasing numbers of orthonormal sections of powers $L^k$ of a positive holomorphic line bundle $L$ on  a compact K\"ahler manifold $M$. In particular, we construct for each positive integer $k$, orthonormal sections $s^k_1,\dots,s^k_{n_k}$ in $H^0(M,L^k)$, $n_k\ge\beta\,\dim H^0(M,L^k)$, such that $\{s^k_j\}$ is a uniformly bounded family, where $\beta$ is an explicit positive constant depending only on the dimension of $M$.  For $m=1$, we can take $\beta=.99564$. 
\end{abstract}

\maketitle


\section{Introduction} In \cite{Bo1}, Bourgain constructed a uniformly bounded orthonormal basis for the Hilbert space of holomorphic polynomials on the 3-sphere $S^3\subset \C^2$. An open question is whether a bounded basis exists in higher dimensions, i.e., on $S^{2m+1}\subset \C^{m+1}$ for $m\ge 2$, or more generally on the boundary of a relatively compact strictly pseudoconvex domain $D$ in a complex manifold $Y$.  An important case, which in fact drives this research, is where $Y$ is the dual bundle $L\inv$  of a positive line bundle $L$ over a compact \kahler manifold $M$, and $\d D\to M$ is the circle bundle consisting of elements of $L\inv$ of length 1; we then seek a bounded orthonormal basis for the space $\hcal^2(\d D)$ of CR holomorphic funcions on $\d D$. Specifically, we identify  $\hcal^2(\d D)$ with the direct sum $\bigoplus_{k=0}^\infty H^0(M,L^k)$ of holomorphic sections of powers of $L$ (see Section~\ref{background}), and we conjecture that there exists a uniformly bounded sequence of orthonormal bases for the spaces $H^0(M,L^k)$, $k\in\Z^+$. Indeed, if $L\to M=\CP^m$ is the hyperplane section bundle, then $\d D = S^{2m+1}$, and the Hilbert space  $\hcal^2(\d D)$ is the $L^2$ completion of the space of polynomials on $\C^{m+1}$ restricted to $S^{2m+1}$. In this case, $H^0(\CP^m,L^k)$ is the space of  homogeneous holomorphic polynomials of degree $k$ on $\C^{m+1}$.

In this paper, we give a partial answer to the question of the existence of uniformly bounded orthonormal bases: 

\begin{theo} \label{main} Let $(L,h)\to (M,\om)$ be a Hermitian holomorphic line bundle over a compact \kahler manifold, with positve curvature $\Theta_h$ and \kahler form  $\om = \frac i2\Theta_h$. Then there exist positive constants $C, \be$  such that for each positive integer $k$, we can find sets of orthonormal holomorphic sections
$$s^k_1,\dots,s^k_{n_k} \in H^0(M,L^k),\qquad n_k \ge \be\dim H^0(M,L^k)\,,$$
such that $ \|s^k_j\|_\infty \le C$ for $1\le j\le n_k$, for all $k\in\Z^+$. \end{theo}

The inner product on $H^0(M,L^k)$ is the $L^2$ inner product induced from the metric $h^k$  on $L^k$ and the volume form $\frac 1{m!}\om^m$ on $M$, where we let $m=\dim M$; see \eqref{induced}.
The sup-norm is given by $ \|s^k\|_\infty=\sup_{z\in M} |s^k(z)|_{h^k(z)}$, for $s^k\in H^0(M,L^k)$.
We recall that \begin{equation}\label{RR} \dim H^0(M,L^k)= \frac 1{m!} c_1(L)^m k^m +O(k^{m-1})\,,\end{equation} by the Riemann-Roch Theorem and Kodaira Vanishing Theorem, so $n_k$ grows at the rate $k^m$.

To compare Theorem \ref{main} with known results, we note that the author and Zelditch showed in \cite{SZlevy} that for $2<p<\infty$, random sections in $H^0(M,L^k)$ of unit $L^2$ norm satisfy a uniform $L^p$ bound independent of $k$ except for rare events of probability $<\exp(-Ck^{2m/p})$. Thus randomly chosen sequences of orthonormal bases will almost surely have uniform $L^p$ bounds for all $p<\infty$.  But random sections  in $H^0(M,L^k)$  of unit $L^2$ norm will have $L^\infty$ norms approximately equal to $\sqrt{m\,\log k}$ with high probability \cite{FZ} (see also \cite{SZlevy}), so random sequences will almost surely not be uniformly bounded.

We shall also  give  explicit positive constants $\be_m$ depending only on the dimension $m$ of $M$ such that  Theorem~\ref{main} holds for all $\be<\be_m$. (See Theorem~\ref{beta}.) For example, for $\dim M=1$, there exist uniformly bounded orthonormal sections $s^k_1,\dots,s^k_{n_k} \in H^0(M,L^k)$ with
$$n_k \ge (.99564)\dim H^0(M,L^k)\,,\quad \mbox{for all }\ k\in\Z^+\,.$$

Theorem \ref{beta} also gives upper bounds for $\limsup_{k\to\infty}\left[ \max_{1\le j\le n_k}\|s^k_j\|_\infty\right]$.  The following result is a consequence of Theorem~\ref{beta}:

\begin{theo} \label{cor1} Let $(L,h)\to (M,\om)$ be as in Theorem \ref{main}.  Then there exists a sequence of holomorphic sections $s^k\in H^0(M,L^k)$, $k=1,2,3,\dots$, such that $\|s^k\|_2=1$ for all $k\in\Z^+$ and
$$\limsup_{k\to\infty} \|s^k\|_\infty\le \kappa_m\,\vol(M)^{-1/2}\,,$$ where $\kappa_m$ is a universal constant depending only on  $m=\dim M$.
\end{theo}

For a continuous section $\psi\in\ccal^0(M,L^k)$, one trivially has $\|\psi\|_\infty/\|\psi\|_2 \ge  \vol(M)^{-1/2}$, with equality if and only if $|\psi|_{h^k}$ is constant. Thus, the constant $\kappa_m$ can be regarded as a measure of the asymptotic ``flatness" of the sections $\{s^k\}$.

Applying Theorem \ref{cor1} to the case where $M=\CP^m$ and $L$ is the hyperplane section bundle $\ocal(1)$ with the Fubini-Study metric, so that $H^0(\CP^m,L^k)$ can be identified with the space of homogeneous holomorphic polynomials  of degree $k$ on $\C^{m+1}$, with the $L^p$ norm of a section given by the $L^p$ norm of the corresponding polynomial over the unit sphere $S^{2m+1}\subset \C^{m+1}$, we obtain the following result: 

\begin{cor} \label{corpoly} For all $m\ge 1$, there exists a sequence of homogeneous holomorphic polynomials $p_k$  on $\C^{m+1}$ such that $\deg p_k=k$ and  $$\sup_k \frac{\|p_k\|_{L^\infty(S^{2m+1})}}{\|p_k\|_{L^2(S^{2m+1})}}<+\infty\,.$$\end{cor}

The result of Bourgain \cite{Bo1} mentioned at the beginning of this paper gives uniformly bounded sequences of orthonormal bases for the spaces $H^0(\CP^1,L^k)$ (which we identify with the spaces of holomorphic homogeneous polynomials on $\C^2$). These bases are of the form
\begin{equation}\label{osc}s^k_j=\frac 1{\sqrt{k+1}}\sum_{q=0}^{k}e^{2\pi ijq/(k+1)}\,\sigma_q\, \chi^k_q\,, \quad 1\le j \le k+1\,,\end{equation}
where $\chi^k_0,\chi^k_1,\dots,\chi^k_k$ are the $L^2$ normalized monomials in $H^0(\CP^1,L^k)$, and $\sigma_q=\pm 1$. (A deep part of the argument in \cite{Bo1} is to choose the signs of the $\sigma_q$ appropriately to obtain uniform bounds.)  Our method is to begin with coherent states peaked at ``lattice points" in the manifold $M$ in place of the monomials $\chi^k_q$ which peak along circles in $\CP^1$.  While the monomials are orthogonal, the coherent states are only approximately orthogonal. So we then modify the coherent states to make them orthogonal before constructing our orthonormal sections $s_j^k$ as oscillating sums of the form \eqref{osc} (but without the $\sigma_q$).

Corollary \ref{corpoly} implies the following result on spherical harmonics:

\begin{cor} \label{coreigen} Let $m\ge 1$, and let $\De$ denote the Laplacian on the round sphere $S^{2m+1}$. Then there exists a sequence of (real) eigenfunctions $f_k$ such that $\Delta f_k=\la_k f_k$, where $\la_k=k(k+2m)$ denotes the $k$-th eigenvalue of $\De$,  and  $$\sup_k \frac{\|f_k\|_{L^\infty(S^{2m+1})}}{\|f_k\|_{L^2(S^{2m+1})}}<+\infty\,.$$\end{cor}

\begin{proof}[Proof (assuming Corollary \ref{corpoly})] Let $p_k=u_k+iv_k$ be as in Corollary \ref{corpoly}. Then $u_k$ and $v_k$ are (real) eigenfunctions of the Laplacian on the sphere with eigenvalue $\la_k$.  We can choose $p_k$ such that $\|u_k\|_2\ge \|v_k\|_2$ and thus $\|u_k\|_2\ge \|p_k\|/\sqrt 2$. We  then let $f_k= u_k$.
\end{proof}

There are related open problems concerning the growth of $L^\infty$ norms of eigenfunctions on compact Riemannian manifolds.  For example, as far as we are aware, it
remains unknown if there are uniformly bounded sequences $\{f_k\}$ of  eigenfunctions satisfying $\Delta f_k=\la_k f_k$  with unit $L^2$ norms on even-dimensional round spheres.   VanderKam \cite{Va} showed that  random sequences of  eigenfunctions $f_k$ of unit $L^2$ norm on $S^2$ satisfy the growth condition $\|f_k\|_\infty = O\left((\log k)^2\right)$.   We note that Toth and Zelditch \cite{TZ} observed that the only compact Riemannian manifolds with completely integrable geodesic flow and with  eigenvalues of bounded multiplicity and which carry a uniformly bounded orthonormal basis of eigenfunctions are flat. A condition for manifolds to have less than maximal eigenfunction growth is given in \cite{SoZ}.

The author would like to thank Zhiqin Lu, Dror Varolin, Steve Zelditch, and Junyan Zhu for useful suggestions.

\section{Background}\label{background} We review in this section  background on geometry and the \szego kernel from \cite{BSZ,SZ, SZvar, SZZ}.

We let  $(L,h)$ be a positive Hermitian holomorphic line
bundle over a compact complex manifold $M$ of dimension $m$, as in Theorem \ref{main}. If we let $e_L$ denote a nonvanishing
local holomorphic section over an open set $\Omega\subset M$, then the curvature form
of $(L,h)$ is given locally over $\Omega$ by
$$\Theta_h= -\ddbar \log |e_L|_h^2\;.$$ Positivity of $(L,h)$ means that the
curvature $\Theta_h$ is positive, so that $\om:=\frac
i2 \Theta_h$ is a \kahler form on $M$. 

The Hermitian metric $h$ on $L$ induces Hermitian metrics
$h^k$ on the powers $L^k$ of the line bundle, and we give the space $H^0(M,L^k)$ of global holomorphic sections of $L^k$ the
Hermitian inner product
\begin{equation}\label{induced}\langle s_k, s'_k\rangle = \int_M h^k(s_k,
\overline{s'_k}) \,\frac 1{m!}\om^m\;,\qquad s_k, s'_k \in
H^0(M,L^k)\,,\end{equation} induced by the metrics
$h,\om$.
As in \cite{BSZ, SZ, SZZ},  we lift sections $s_k\in H^0(M,L^k)$ to the circle
bundle
$X{\buildrel {\pi}\over \to} M$  of unit vectors in the dual bundle $L\inv\to M$
endowed with the dual metric $h\inv$. Since $(L,h)$ is positive, $X$ is a
strictly pseudoconvex CR manifold.  The lift $\hat s_k:X\to \C$ of the section
$s_k$ is given by
$$\hat{s}_k(\lambda) = \left( \lambda^{\otimes k}, s_k(z)
\right)\,,\quad\la\in \pi\inv(z)\,.$$  The sections $\hat s_k$ span the space 
$\hcal^2_k(X)$ of CR holomorphic functions $\hat s$ on $X$ satisfying $\hat
s(e^{i\theta}x)= e^{ik\theta}\hat s(x)$. This provides  isomorphisms
$$H^0(M,L^k) \approx \hcal^2_k(X)\,,\qquad s_k\mapsto \hat s_k\,.$$
We shall henceforth identify $H^0(M,L^k)$ with $\hcal^2_k(X)$ by identifying $s_k$ with $\hat s_k$.

The {\it \szego projector\/} of level $k$ is the
orthogonal projector
$\Pi_k:L^2(X)\to\hcal^2_k(X)$, which is given by the {\it
\szego kernel}
$$\Pi_k(x,y)=\sum_{j=1}^{d_k} \wh S^k_j(x)\overline{\wh S^k_j(y)}
\qquad (x,y\in X)\;,$$ where $\{\wh S^k_1,\dots,\wh S^k_{d_k}\}$ is an orthonormal basis of $\hcal^2_k(X)$ and $d_k=\dim \hcal^2_k(X)$.
It was shown in \cite{Ca,Ti,Z} (see also \cite{BBS}) that the \szego kernel on the diagonal has the
asymptotics:
\begin{equation}\label{Ze}\Pi_k(x,x) =  \frac {k^m}{\pi^m}
+O(k^{m-1})\,.\end{equation}  We write \begin{equation}\label{Pk}P_k(z,w):=
\frac{|\Pi_k(x,y)|}{\sqrt{\Pi_k(x,x)}
\sqrt{\Pi_k(y,y)}}\,,\quad x\in\pi\inv (z), \ y\in\pi\inv(w)\,.\end{equation} 

We shall apply the following  off-diagonal asymptotics of this normalized \szego kernel:

\begin{prop}\label{DPdecay}  Let $b,\ q\in\R^+$.
Then 
$$ P_k(z,w) =\left\{\begin{array}{ll}   e^{-\frac
k2\,[1+o(1)]\,\dist(z,w)^2}\,,\quad &\mbox{uniformly for }\ \dist(z,w)\le b\,\sqrt{\frac
{\log k}{k}} \\[8pt] O(k^{-q})\,, & \mbox{uniformly for }\ \dist(z,w)\ge\sqrt{(2q+2m+1)\frac
{\log k}{k}}\end{array}\right. \;.$$
\end{prop}

The estimates in  Proposition \ref{DPdecay} follow from Propositions 2.6--2.8 of \cite{SZvar}. (See also \cite{BBS, MM}. In fact, one has the bound $P_k(z,w)=O(e^{-C\sqrt{k}\,\dist(z,w)})$  \cite{Ch, De, Li}, which is sharper than the above when $\dist(z,w)>k^{-1/2+\epsilon}$, and the first estimate of the proposition holds for $\dist(z,w)<k^{-1/3} $ \cite{SZ}, but the estimates of the proposition suffice for our purposes.)  

We shall apply Proposition \ref{DPdecay} with $q=m+1$.

\section {proof of Theorem \ref{main}} \label{theproof} 

The uniformly bounded sections we construct  are the opposite of peak sections which maximize the sup norm.  We can think of these bounded sections as ``flat sections" since they lack large peaks.  They will be constructed as linear combinations of peak sections centered at ``lattice points" in the following steps: \begin{enumerate}

\item Construct peak sections at lattice-like points; these sections will be approximately orthonormal;

\item modify the sections to be orthonormal;

\item construct a family of linear combinations of these modified sections so they will be orthonormal and uniformly bounded.

\end{enumerate}
\subsection{Step 1: approximately orthonormal peak sections} Here we follow the method in \cite{SZZ} based on the \szego kernel asymptotics of Proposition \ref{DPdecay}.  We repeat the argument here, since we need slightly sharper estimates.

Choose a point $z_0\in M$ and identify $T_{z_0}M$ with $\R^{2m}$. Let $$C_t:=
[-t,t]^{2m}\subset \R^{2m} \equiv T_{z_0}M$$
denote the $2m$-cube of width $2t$ centered at the origin.  Let $\ga>1$ be arbitrary, and choose $t$
sufficiently small so that 
\begin{equation}\label{distance}\ga\inv \|v-w\|\le
\dist(\exp_{z_0} (v), \exp_{z_0}(w)) \le \ga \|v-w\|\,,
\quad\mbox{for }\ v,w\in C_{2t}\;.\end{equation}  For each $k>0$,  we construct a
lattice of points $\{z^k_\nu\}$ in $M$ as follows:  Let
\begin{equation}\label{grid} \Ga_k=\left\{(\nu_1,\dots,\nu_{2m})\in
\Z^{2m}:|\nu_j|\le
\frac {t\,{\sqrt k}}{a}\right\}\,,\end{equation}
where $a$ is to be chosen later.  The number $n_k$ of points in $\Ga_k$ is given by
\begin{equation}\label{n} n_k = \left(2\left\lfloor
\frac {t\,\sqrt k}{a}\right\rfloor+1\right)^{2m} = \left( \frac
{2t}{a}\right)^{2m}k^m+ O(k^{m-1/2})\,.\end{equation} 
It follows from  \eqref{n} and the Riemann-Roch  theorem \eqref{RR} that $ n_k \ge \be\dim H^0(M,L^k)$ for some positive consant $\be$.

We now begin our construction of $n_k$ orthonormal sections in $ H^0(M,L^k)\equiv \hcal^2_k(X)$: We let
\begin{equation}\label{lattice} z^k_\mu = \exp_{z_0}\left( \frac a{\sqrt k}
\,\mu\right)\in \exp_{z_0}(C_t),\quad \mbox {for } \ \mu\in \Ga_k\,.\end{equation}
We choose points
$y^k_\mu\in X$ with
$\pi(y_\mu)= z_\mu$,  where we omit the
superscript $k$ to simplify notation.  We consider the $L^2$-normalized {\it coherent states\/}
\begin{equation} \label{coherent} \Phi^k_\mu(x):=  \frac{\Pi_k(x,y_\mu)}{\sqrt {\Pi_k(y_\mu,y_\mu)}} \qquad (\mu\in\Ga_k)\,,\end{equation}
at the lattice points $z_\mu$.  These sections are ``almost" orthonormal, since
\begin{equation}\label{almostortho} \lla \Phi^k_\mu,\Phi^k_\nu\rra = \frac{\int\Pi_k(x,y_\mu)\Pi_k(y_\nu,x)\,dx}{\sqrt {\Pi_k(y_\mu,y_\mu)}\sqrt {\Pi_k(y_\nu,y_\nu)}}= \frac{\Pi_k(y_\nu,y_\mu)}{\sqrt {\Pi_k(y_\mu,y_\mu)}\sqrt {\Pi_k(y_\nu,y_\nu)}}\,,\end{equation} and thus
\begin{equation}\label{inner} \left|\lla \Phi^k_\mu,\Phi^k_\nu\rra \right| = P_k(z_\mu,z_\nu)\,,\end{equation} which decays rapidly for $\mu\neq\nu$, thanks to Proposition \ref{DPdecay}.

\subsection{Step 2:  orthonormal peak sections}\label{step2}
 The next step is to modify the set $\{\Phi^k_\nu\} $ of coherent states to obtain an orthonormal set $\{\Psi^k_\nu\}$.
As in \cite{SZZ}, we consider the Hermitian $n_k\times  n_k$ matrices $$\De_k = (\De_{\mu\nu}), \qquad
\De_{\mu\nu}=\lla \Phi^k_\mu,\Phi^k_\nu\rra,\qquad  \mbox{for }\ \mu,\nu\in\Ga_k
\,.$$
 We note that by \eqref{almostortho}, the diagonal entries of $\De_k$ are 1.  Since $|\De_{\mu\nu}|=P_k(z_\mu,z_\nu)$, 
Proposition
\ref{DPdecay} (with $q=m+1$) says that
\begin{eqnarray}\label{near} |\De_{\mu\nu}| &\le &
e^{[-1+o(1)] \,\frac k2\,\dist(z_\mu,z_\nu)^2}, \quad  \mbox{if }\ \dist(z_\mu,z_\nu)
\le b\sqrt{\frac {\log k}k}\,, \\[8pt] |\De_{\mu\nu} |&=&O( k^{-m-1}), \qquad  \mbox{if }\
\dist(z_\mu,z_\nu)
\ge b\sqrt{\frac {\log k}k}\label{far}\,,\end{eqnarray}
where $b=\sqrt{4m+3}$.

It was shown in \cite[p.~1987]{SZZ} 
that for all $\eta>0$, we can choose the constant $a$ in \eqref{lattice} such 
that
\begin{equation}\label{almost} \max_{\mu\in\Ga_k}\textstyle\left( \sum_{\nu\in\Ga_k\sm\{\mu\}}| \De_{\mu\nu}|\right) \le  \eta
\,,\quad \mbox{for }\ k\gg 0\,. \end{equation}  We give below a simplified proof of \eqref{almost}, which yields an estimate for $a$:

Fix an element
$\mu_0\in\Ga_k$. By \eqref{n} and 
\eqref{far}, we have
$$ \sum_{\nu\in\Ga_k\sm\{\mu_0\}} |\De_{\mu_0\nu}| = \sum_{\nu\in\Ga_k(\mu_0)} |\De_{\mu_0\nu}| +  O(k\inv) \,,$$ where
$$\Ga_k(\mu_0) =\left\{\nu\in\Ga_k: 0<\dist(z_{\mu_0},z_\nu)
\le \textstyle b\sqrt{\frac {\log k}k}\right\}\,.$$
Let $\ep>0$ be arbitrary.  By \eqref{near},
\begin{equation*}  \sum_{\nu\in\Ga_k(\mu_0)} |\De_{\mu_0\nu}| \le \sum_{\nu\in\Ga_k(\mu_0)}
e^{-\frac k2\,(1-\ep) \,\dist(z_{\mu_0},z_\nu)^2},\quad \mbox{for }\ k\gg 0\,.\end{equation*}
Now let $a'=a/\ga$, and let $\tilde a = a'\sqrt{1-\ep}$.
Since $\dist(z_\nu,z_{\mu_0})>\frac {a'}{\sqrt k}\|\nu-{\mu_0}\|$, we then have by \eqref{near}
\begin{eqnarray*}\sum_{\nu\in\Ga_k(\mu_0)}|\De_{\mu_0\nu}|& \le &\sum_{\nu\in\Ga_k(\mu_0)}  e^{-\tilde a^2\|\nu-\mu_0\|^2/2}
\ =\  \sum_{\nu\in\Z^m\sm{0}} e^{-\tilde a^2\|\nu\|^2/2} = \left[ \sum_{j=-\infty}^\infty e^{-\tilde a^2j^2/2}\right]^{2m}-1\\ &< & 
\left[ 1+\int_{-\infty}^\infty e^{-\tilde a^2x^2/2}\,dx\right]^{2m}-1   = \left(1+\tilde a\inv\sqrt{2\pi}\right)^{2m}-1\,,
\end{eqnarray*} for $k\gg 0$. Thus   \eqref{almost} holds whenever
\begin{equation}\label{atilde} \left(1+ \tilde a\inv\sqrt{2\pi}\right)^{2m}\le  1+\eta\,.\end{equation}

Recall that the $\ell^\infty\to\ell^\infty$ mapping norm of  a linear map $A\in\mbox{Hom}(\C^n,\C^n)$ is given by  \begin{equation}\label{mapnorm1}\|A\|_{\ell^\infty\to\ell^\infty} = \sup\{\|Av\|_{\ell^\infty}: v\in\C^n,\, \|v\|_{\ell^\infty}=1\}
=\max_{1\le \mu\le n}\sum_{\nu=1}^n|A_{\mu\nu}|\,,\end{equation}
where $\|v\|_{\ell^\infty} = \max_{1\le \mu\le n}|v_\mu|$. 
By \eqref{almost}--\eqref{mapnorm1}, we have:

\begin{lem}\label{mapnorm}Let $\De_k=I-A_k$. Suppose that 
\begin{equation}\label{abound} a >  \frac{\ga\,\sqrt{2\pi}}{(1+\eta)^{1/2m}-1}\,,\end{equation} where $\ga$ satisfies the distortion bound \eqref{distance}. Then  \begin{equation}\label{mapnormA} \limsup_{k\to\infty}\|A_k\|_{\ell^\infty\to\ell^\infty} <\eta\,.\end{equation} \end{lem}

In the following, we let $0<\eta<1$, and we let $a$ satisfy \eqref{abound}, so that $\|A_k\|_{\ell^\infty\to\ell^\infty}\le\eta<1$ for $k\gg 0$.  It follows that the eigenvalues of $\De_k$ are bounded below by $1-\eta$, and therefore
$\De_k$ is invertible for $k\gg 0$.   From the Taylor series
$$(1-x)^{-1/2} = \sum_{j=0}^\infty\frac {(2j)!}{4^j\,j!^2}\,x^j,$$ it follows that  the (positive definite Hermitian) square root of $\De_k\inv$ is given by
\begin{equation}\label{series} \De_k^{-1/2}=I+ \sum_{j=1}^\infty\frac {(2j)!}{4^j\,j!^2}\,A_k^j\,,\end{equation} where the series converges in the $\ell^\infty\to\ell^\infty$ mapping norm, for $k\gg 0$. Furthermore by \eqref{mapnormA},
\begin{eqnarray}\label{bound2}\|\De_k^{-1/2}\|_{\ell^\infty\to\ell^\infty} &\le & 1+ \sum_{j=1}^\infty\frac {(2j)!}{4^j\,j!^2}\,\|A_k^j\|_{\ell^\infty\to\ell^\infty}\ \le\ \sum_{j=0}^\infty\frac {(2j)!}{4^j\,j!^2}\,\|A_k\|^j_{\ell^\infty\to\ell^\infty}\notag\\& =& \left(1-\|A_k\|_{\ell^\infty\to\ell^\infty}\right)^{-1/2}\ \le\ (1-\eta)^{-1/2}\,, \quad \mbox{for }\ k\gg 0\,.\label{mapnormB}\end{eqnarray}

We write $\De_k^{-1/2}=B_k=(B_{\mu\nu})$. To complete Step 2, we define the ``quasi-coherent states" \begin{equation}\label{ortho} \Psi^k_\mu = \sum_{\nu\in\Ga_k} B_{\mu\nu}\Phi^k_\nu \in H^0(M,L^k)\,,\quad \mbox{for } \mu\in\Ga_k.\end{equation}
By definition, $$\langle \Psi_{\mu}^k, \Psi_{\nu}^k\rangle = \sum_{\rho,\sigma\in\Ga_k} B_{\mu\rho}\bar B_{\nu\sigma}\De_{\rho\sigma}= \de_\mu^\nu\,,$$ and thus the $\Psi_{\mu}^k$ are orthonormal.

\subsection{Step 3:  orthonormal flat sections}\label{step3} Our orthonormal uniformly bounded (``flat") sections $\{\s_j^k\}$ are easily constructed from our orthonormal peak sections $\Psi_{\mu}^k$:  Let $\zeta=e^{2\pi i/{n_k}}$ be a primitive $(n_k)$-th root of unity, and let $\tau^1, \tau^2,\dots,\tau^{n_k}$ be the lexicographic (or any other) ordering of the elements of $\Ga_k$. We let 
\begin{equation}\label{orthoflat} s_j^k= \frac 1{\sqrt{n_k}}\sum_{q=1}^{n_k} \zeta^{qj}\Psi_{\tau^q}^k\,,\qquad 1\le j\le n_k\,.\end{equation}
Since the $\Psi_{\tau^q}^k$ are orthonormal, $$\langle s_j^k, s_l^k\rangle = \frac 1{n_k} \sum_{q=1}^{n_k} \zeta^{(j-l)q} = \de_l^j\,,\qquad \mbox{for }\ 1\le j,l\le n_k\,,$$ and thus $\{s_j^k\}$ is an orthonormal family.
To verify that the $s_j^k$ are uniformly bounded, we consider the linear maps
$$F_k:\C^{n_k}\to \hcal^2_k(X)\,,\qquad (v_1,\dots,v_{n_k})\mapsto \sum_{j=1}^{n_k} v_j\,\Phi^k_{\tau^j}\,,$$
with the $\ell^\infty\to L^\infty(X)$ mapping norm 
\begin{eqnarray*}\|F_k\|_{\ell^\infty\to L^\infty(X)} =
\sup \left\{\sup_{x\in X}\left| \!\!\sum_{\ \mu\in\Ga_k} v_\mu\,\Phi^k_\mu(x)\right|:  |v_\mu|\le 1,\ \mbox{for }\ \mu\in\Ga_k \right\}=\sup_{x\in X}\sum_{\mu\in\Ga_k} |\Phi^k_\mu(x)|\,.\end{eqnarray*}

\begin{lem}\label{mapnormF} \ 
$\|F_k\|_{\ell^\infty\to L^\infty(X)} \le c\,{k^{m/2}}\,, 
$ for some constant $c<+\infty$.\end{lem}
\begin{proof}
Let $x\in X$ be arbitrary, and let $z=\pi(x)\in M$.  Then by \eqref{Ze}, \begin{equation}\label{Phik}\Phi^k_{\mu}(x) =  \Pi_k(x,x)^{1/2}P_k(z,z_\mu)= \left(\pi^{-m/2}+o(1)\right)k^{m/2}\,P_k(z,z_\mu). \end{equation}

We consider two cases:

{\it Case 1: $z\notin \exp_{z_0}(C_{2t})$.}  Then $d(z,z^k_\mu)\ge \dist(z,C_{t})\ge t/\ga$ for all $\mu$ and $k$, so by  
\eqref{far}, $P_k(z,z^k_\mu)=O(k^{-m-1})$ and hence by \eqref{n} and \eqref{Phik}, $\sum_\mu|\Phi^k_{\mu}(x)|=o(k^{m/2})$ uniformly for $x\in X\sm \pi\inv\left(\exp_{z_0}(C_{2t})\right)$.

{\it Case 2: $z=\exp_{z_0}(p)$, where $ p\in C_{2t}$.} As before, we suppose that $\ep>0$ and we let $a'=a/\ga$, $\tilde a = a'\sqrt{1-\ep}$. As in the proof of \eqref{almost}, $$\sum_{\mu\in\Ga_k} P_k(z,z_\mu) \le \sum\left\{
e^{-\frac k2\,(1-\ep) \,\dist(z,z_\mu)^2}:  \dist(z,z_\mu)
\le \textstyle b\sqrt{\frac {\log k}k}\right\}  +O(k^{-1})\,.$$
By \eqref{distance}, $$\dist(z,z_\mu) \ge \ga\inv\left\|p-\frac {a}{{\sqrt k}}\mu\right\| = \frac {a'}{\sqrt k}\,\|p^k-\mu\|\,, \qquad p^k=\frac{\sqrt k} a\, p\,.$$
Therefore \begin{multline*} \sum_{\mu\in\Ga_k}e^{-\frac k2\,(1-\ep)\,\dist(z,z_\mu)^2}  \le \sum_{\mu\in\Z^{2m}} e^{-\tilde a^2\|\mu-p^k\|^2/2} \le  \sum_{\mu\in\Z^{2m}} e^{-\tilde a^2\|\mu\|^2/2} \\ <
\left[ 1+\int_{-\infty}^\infty e^{-\tilde a^2x^2/2}\,dx\right]^{2m}   = \left(1+\tilde a\inv\sqrt{2\pi}\right)^{2m}\,,
\end{multline*} where the second inequality is by the Poisson summation formula applied to the function $f(x)=e^{-\tilde a^2x^2/2}$. 
The estimate of the lemma then follows from \eqref{Phik}.\end{proof}

\begin{proof}[Completion of the proof of Theorem \ref{main}]It remains to show that the $s^k_j$ are uniformly bounded.  Combining \eqref{ortho}--\eqref{orthoflat}, we have
$$s_j^k = \frac 1{\sqrt{n_k}}\sum_{q=1}^{n_k} \sum_{\nu\in\Ga_k}\zeta^{qj} B_{\tau^q\nu}\Phi^k_\nu\,.$$ Fix $j$ and let 
$$v_\nu= \sum_{q=1}^{n_k}\zeta^{qj}  B_{\tau^q\nu}\,, \quad\mbox{for } \nu\in\Ga_k.$$
By Lemma \ref{mapnorm} and  \eqref{mapnormB}, \begin{equation}\label{vnu}|v_\nu|\le \sum_{q=1}^{n_k} |B_{\tau^q\nu}| = \sum_{q=1}^{n_k} |B_{\nu\tau^q}|= \|\De_k^{-1/2}\|_{\ell^\infty\to\ell^\infty} \le (1-\eta)^{-1/2}\,,\end{equation} where $\eta=(1+\tilde a\inv\sqrt{2\pi})^m-1$. Thus by Lemma \ref{mapnormF} and \eqref{vnu}, $$\|s_j^k\|_{L^\infty(X)} = \frac 1{\sqrt {n_k}} \left\| \sum_{\nu\in\Ga_k}\Phi^k_\nu\,v_\nu\right\|_{L^\infty(X)} \le \frac 1{\sqrt {n_k}} \|F_k\|_{\ell^\infty\to L^\infty(X)}\;\|B_k\|_{\ell^\infty\to\ell^\infty} \le \frac {c}{\sqrt{1-\eta}}\,\frac {k^{m/2}}{\sqrt {n_k}}\,,$$ for $k$ sufficiently large. By \eqref{n}, $n_k\ge \be' k^m$ for a positive constant $\be'$, and therefore the  $s_j^k$ are uniformly bounded above.
\end{proof}

\section{Universal bounds}

In this section, we modify the above argument to obtain a  universal value (depending  only on the dimension of $M$) of the fraction $\be$ in Theorem \ref{main}.  We also give a universal bound for  the asymptotic sup norms of the  orthonormal sections:

\begin{theo} \label{beta} There exist positive constants $\be_m$  depending only on $m\in\Z^+$ such that if   $(L,h)\to (M,\om)$ is as in Theorem \ref{main}, with  $\dim M=m$, then for all $\be< \be_m$, there exist sets of orthonormal holomorphic sections
$$s^k_1,\dots,s^k_{n_k} \in H^0(M,L^k)\,,\qquad  n_k \ge \be\,\dim H^0(M,L^k)\,,$$
such that the family $\{s^k_j:1\le j\le n_k,\,k\in \Z^+\}$ is uniformly bounded. 

Furthermore,  there exist constants $\kappa_m(\beta)$, depending only on $m$ and $\be$, such that  the $s_j^k$ can be chosen to satisfy the universal asymptotic $L^\infty$ bound \begin{equation}\label{flat}\limsup_{k\to\infty}\left[ \max_{1\le j\le n_k}\|s^k_j\|_\infty\right] \le \kappa_m(\beta)\,\vol(M)^{-1/2}\,.\end{equation}
\end{theo}

\begin{proof} Let $a_m,\be_m\in \R^+$ be given by \begin{equation}\label{sigma}\sum_{j=-\infty}^{+\infty} e^{-a_m^2\,j^2/2}=2^{1/2m}\,,\qquad \be_m= {\pi^m}/{a_m^{2m}}\,.\end{equation}  Suppose that $\be<\be_m$, and choose $a>a_m$ such that  $\be<{\pi^m}/{ a^{2m}}$.  Then choose $\tilde a>a_m$ and $\ga>1$ such that $a_m<\tilde a <a/{\ga}$. We decompose $M$ into a finite number of disjoint domains $\{U_j\}_{1\le j\le q}$ with piecewise smooth boundaries  such that $M=\bigcup_{j=1}^q\overline U_j$ and that there exist points $p_j\in U_j$ and open sets $W_j\subset T_{p_j}M\equiv \R^{2m}$ such that  $\exp_{p_j}(W_j)=U_j$,
\begin{equation}\label{dist}\ga\inv \|v-w\|\le
\dist(\exp_{p_j} (v), \exp_{p_j}(w)) \le \ga \|v-w\|\,,
\quad\mbox{for }\ v,w\in W_j\;,\end{equation} and writing $\exp^*_{p_j}(\frac 1{m!}\om^m) = g_{(j)}\,dx_1\wedge\cdots\wedge dx_{2m}$, 
\begin{equation}\label{estvol} 0 < g_{(j)}(v)\le 1 +\frac\de{\vol(M)}\,,\quad\mbox{for }\ v\in W_j\,,\end{equation}
where $\de>0$ is to be chosen later.

Choose smooth  domains $U''_j\subset\!\subset U'_j\subset\!\subset U_j$ such that $\vol(M\sm\bigcup U''_j)<\de$, and let $W_j''=\exp_{p_j}\inv(U_j'')$. By \eqref{estvol}, $$\vol(U_j'') = \int_{W_j''} g_{(j)}\,dx_1\wedge\cdots\wedge dx_{2m}
\le \left(1 +\frac\de{\vol(M)}\right)\vol(W_j'')\,,$$ and therefore
\begin{equation}\label{volbound}\sum_{j=1}^q \vol(W_j'') \ge \left(1 -\frac\de{\vol(M)}\right)\sum_{j=1}^q \vol(U_j'') \ge \vol(M)-2\,\de\,.\end{equation}

For $k\in\Z^+$, we write
\begin{equation}\label{latticej} z^k_{\mu j} = \exp_{p_j}\left( \frac a{\sqrt k}
\,\mu\right)\quad \mbox {for } \ \mu\in \Z^{2m}\,,\ 1\le j\le q\,,\end{equation}
and we let  $$\Ga_{kj}= \left\{\mu\in\Z^{2m}:  \frac a{\sqrt k}\,\mu \in W_j''\right\} 
=\{\mu\in\Z^{2m}:z^k_{\mu j}\in U_j''\}\,.$$ We thus obtain a collection of ``lattice points" $\{z_{\mu j}^k:\mu\in \Ga_{kj},1\le j\le q\}$ throughout $M$.
As before, we choose points
$y^k_{\mu j}\in X$ with
$\pi(y^k_{\mu j})= z^k_{\mu j}$, and we
consider the family of coherent states
$$\Phi^k_{\mu j}(x):= \frac{\Pi_k(x,y^k_{\mu j})}{ {\Pi_k(y^k_{\mu j},y^k_{\mu j})}^{1/2}}\in \hcal^2_k(X)\,,\qquad \mu\in\Ga_{kj},\ j=1,\dots,q\,,$$ at these lattice points.  It follows from \eqref{volbound} that the number $n_k$ of lattice points satisfies the inequality \begin{equation}\label{nk}n_k = \sum_{j=1}^q \# (\Ga_{kj})=\frac {k^m}{a^{2m}}\,\sum_{j=1}^q\vol(W_j'') +O(k^{m-1/2}) > 
\frac {\vol(M)-3\de}{a^{2m}}\,k^m \,, \quad \mbox{for } k\gg 0\,.\end{equation}

To show that $n_k$ satisfies the lower bound of the theorem, we recall that the volume of $M$ is given by \begin{equation}\label{volume} \vol(M)=\int_M\frac 1{m!}\om^m = \frac 1{m!}\int_M [\pi\,c_1(L,h)]^m = \frac{\pi^m}{m!}c_1(L)^m\,.\end{equation}
Let $\ep= \pi^m/a^{2m}-\be>0$. Then by  \eqref{RR} and \eqref{nk},
$$\frac {n_k}{\dim H^0(M,L^k)} > \frac {m!\,[\vol(M)-3\de]}{c_1(L)^m\,a^{2m}} =\frac{\pi^m-3\,m!\,\de/c_1(L)^m}{a^{2m}}= \be+\ep - \frac{3\,m!}{a^{2m}\,c_1(L)^m}\,\de\,,$$ for $k\gg 0$.  Choosing $\de \le  \frac 1{3\,m!}a^{2m}c_1(L)^m\ep$, we obtain the desired bound \begin{equation}\label{>}n_k > \be\,\dim H^0(M,L^k)\,.\end{equation}

We now construct $n_k$ orthonormal sections of $\hcal^2_k(X)=H^0(M,L^k)$ for $k$ sufficiently large. Following the approach of Section~\ref{step2}, we define the Hermitian $n_k\times  n_k$ matrices $$\De_k = (\De_{(\mu j)(\nu l)}), \qquad
\De_{(\mu j)(\nu l)}=\lla \Phi^k_{\mu j},\Phi^k_{\nu l}\rra,\qquad  \mbox{for }\ \mu\in\Ga_{kj},  \ \nu\in \Ga_{kl},\ 1\le j,l\le q
\,.$$
Recalling that $\tilde a>a_m$, we let $$\eta=\left(\sum_{j=-\infty}^{+\infty} e^{-\tilde a^2\,j^2/2}\right)^{2m}-1 <1\,.$$
Fix $\mu_0,j_0$, with $\mu_0\in\Ga_{kj_0}$. Since
$\tilde a <a/\ga$, we see by the argument in Section~\ref{step2} that  
$$ \sum_{(\mu,j)\neq(\mu_0,j_0)}\!\! \De_{(\mu_0 j_0)(\mu j)} \le \sum_{\nu\in\Z^m\sm{0}} e^{-\tilde a^2\|\nu\|^2/2}+O(k^{-1}) = \eta +O(k^{-1}).$$  Therefore, for $k$ sufficiently large, $\De_k$ is invertible and we have
$$\|\De_k^{-1/2}\|_{\ell^\infty\to\ell^\infty} \le (1-\eta)^{-1/2} +O(k^{-1})\,.$$ We can then construct as before the orthonormal family
\begin{equation}\label{ortho2} \Psi^k_{\mu j} = \sum_{\nu l} [\De_k^{-1/2}]_{(\mu j)(\nu l)}\Phi^k_{\nu l} \in \hcal^2_k(X)\,, \quad\mbox{for }\ k\gg0\,.\end{equation}

We let $\tau^1,\tau^2,\dots,\tau^{n_k}$ be an (arbitrary) ordering of the indices $(\mu j)$.
As in Section \ref{step3}, we define the orthonormal sections 
\begin{equation}\label{same} s_j^k= \frac 1{\sqrt{n_k}}\sum_{q=1}^{n_k} \zeta^{qj}\Psi_{\tau^q}^k\,,\quad 1\le j\le n_k\,.\end{equation}
and we consider the linear maps
$$F_k:\C^{n_k}\to \hcal^2_k(X)\,,\qquad (v_1,\dots,v_n)\mapsto \sum_{j=1}^{n_k} v_j\,\Phi^k_{\tau^j}\,.$$
By the proof of Lemma \ref{mapnormF} (with  $C_{2t}$ replaced with $\bigcup U_j'$), we conclude that 
\begin{equation*} \|F_k\|_{\ell^\infty\to L^\infty(X)} \le (1+\eta+o(1))\,\frac{k^{m/2}}{\pi^{m/2}}\,. 
\end{equation*} for $k$ sufficiently large.  It then follows as before that $$\|s_j^k\|_{L^\infty(X)} \le  \frac 1{\sqrt {n_k}}\|F_k\|_{\ell^\infty\to L^\infty(X)}\;\|\De^{-1/2}_k\|_{\ell^\infty\to\ell^\infty} \le \frac {1+\eta+o(1)}{\pi^{m/2}\sqrt{1-\eta}}\,\frac {k^{m/2}}{\sqrt{n_k}}
\,,$$ for $k$ sufficiently large. Since $\dim H^0(M,L^k)=\frac 1{\pi^m}\vol(M)[k^m+o(k^{m-1})]$, it follows from \eqref{>} that
$$n_k>\frac\beta{\pi^m}\,\vol(M)\,k^m\,,\quad \mbox{for }\ k\gg 0\,,$$ and therefore
\begin{equation}\label{flatb}\|s_j^k\|_{L^\infty(X)} < \frac{1+\eta}{\sqrt{\beta(1-\eta)}}\,\vol(M)^{-1/2}\,,\quad \mbox{for }\ k\gg 0\,.\end{equation}

\end{proof}

\begin{rem} The bound in \eqref{flatb} depends on the choice of $\tilde a<a/\ga<a<\pi^{1/2}/\be^{1/2m}$.  However, if one chooses a sequence $\tilde a(\nu)\nearrow \pi^{1/2}/\be^{1/2m}$,  then from the resulting sequences $\{s^k_j(\nu)\}$  we can construct $\{s^k_j\}$  satisfying \eqref{flatb} with 
$$\eta=\left[\sum_{j=-\infty}^{+\infty}\exp\left(-\frac{\pi j^2}{2\,\be^{1/m}}\right)\right]^{2m}-1\,.$$

\end{rem}

\subsection{Numerical values for $\be_m$.}\label{numerical} Solving
\eqref{sigma} numerically using {\it Maple 18\/}, we obtain the following values (to 5 decimal places):
$\be_1 \approx .99220$,
$\be_2\approx .44342$,
$\be_3\approx .17782$,
$\be_4 \approx .06630$,
$\be_5 \approx .02345$,
$\be_6\approx .00796$.

However, one obtains larger values of the constants $\be_m$ in Theorem~\ref{beta} by using the lattice points 
\begin{equation}\label{lattice2} z^k_{\mu j} = \exp_{p_j}\left[ \frac a{\sqrt k}
\,(\mu_1+e^{i\pi/3}\mu_2,\; \mu_3+e^{i\pi/3}\mu_4,\;\dots,\mu_{2m-1}+e^{i\pi/3}\mu_{2m})\right],\end{equation}  for $\mu=(\mu_1,\dots,\mu_{2m})\in \Z^{2m}$
(where we identify $T_{p_j}M\equiv R^{2m}\equiv\C^m$),  instead of the points of the lattice \eqref{latticej}.  In place of \eqref{nk}, we have
\begin{equation*}n_k >\left( \frac 2{\sqrt 3}\right)^m\;
\frac {\vol(M)-3\de}{a^{2m}}\,k^m \,, \quad \mbox{for } k\gg 0\,.\end{equation*}
We note that $|\mu_1+e^{i\pi/3}\mu_2|^2=\mu_1^2+\mu_2^2+\mu_1\mu_2$; then in place of \eqref{sigma}, we let $\al_m,\be'_m\in\R^+$ be given by 
\begin{equation}\label{bettersigma}\sum_{\mu\in\Z^2} e^{-\al_m^2(\mu_1^2+\mu_2^2+\mu_1\mu_2)/2}=2^{1/m}\,,\qquad \be'_m= \left(\frac{2}{\sqrt{3}}\right)^m\,\frac {\pi^m}{\al_m^{2m}} = \left(\frac{2\pi}{\sqrt{3}\,\al_m^2}\right)^m\,.\end{equation}

Let $\be<\be'_m$, and choose $\ga,\al,\tilde \al$ such that  $\al_m<\tilde \al<\al/{\ga}$ and $\be<\left(\frac{2\pi}{\sqrt{3}\,\al^2}\right)^m$. Repeating the proof of Theorem \ref{beta}, we  obtain the estimate $$ \sum_{(\mu,j)\neq(\mu_0,j_0)}\!\! \De_{(\mu_0 j_0)(\mu j)} \le \left[\sum_{\mu\in\Z^2} e^{-\tilde \al^2(\mu_1^2+\mu_2^2+\mu_1\mu_2)/2}\right]^m-1+O(k^{-1})\,.$$
We then conclude that
$$\frac {n_k}{\dim H^0(M,L^k)} > \left(\frac 2{\sqrt 3}\right)^m\,\frac {\vol(M)-3\de}{c_1(L)^m\,\al^2} = \left(\frac 2{\sqrt 3}\right)^m\,\frac{\pi^m-3\,\de/c_1(L)^m}{\al^{2m}}>\be\,,$$ for $k\gg 0$, if $\de$ is chosen small enough.  Thus Theorem~\ref{main} holds for all $\be$ less than the value of $\be'_m$ given by \eqref{bettersigma}. Solving \eqref{bettersigma} numerically  using {\it Maple 18\/}, one obtains the better values (to 5 decimal places):
$\be'_1 \approx  .99564$,
$\be'_2 \approx   .45867$,
$\be'_3 \approx   .19254$,
$\be'_4 \approx   .07572$,
$\be'_5 \approx   .02838$,
$\be'_6 \approx    .01024$.
In dimension 1, the lattice \eqref{lattice2} appears to give the best value of $\be_1$ (compared with other lattices), but  this lattice is probably not optimal for $m\ge 2$.
An open question is whether the result can be improved further by using coherent states at other collections of points, e.g., Fekete points.


\begin{thebibliography}{WWW}

 \bibitem[BBS]{BBS} R. Berman, B. Berndtsson and J. Sj\"ostrand, A direct
approach to Bergman kernel asymptotics for positive line bundles,
 {\it Ark.\ Mat.} 46 (2008), 197--217. 

\bibitem[BSZ]{BSZ} P. Bleher, B. Shiffman and S. Zelditch, Universality and
scaling of correlations between zeros on complex manifolds,  {\it Invent.\ Math.}
142 (2000), 351--395.


\bibitem[Bo]{Bo1} J. Bourgain,  Applications of the spaces of homogeneous polynomials to
some problems on the ball algebra. Proc.\ Amer.\ Math.\ Soc.\ 93 (1985), 277--283.


\bibitem[Ca]{Ca} D. Catlin, The Bergman kernel and a theorem of Tian, in: {\it
Analysis and Geometry in Several Complex Variables\/}, G. Komatsu and M.
Kuranishi, eds., Birkh\"auser, Boston, 1999.

\bibitem[Ch]{Ch} M. Christ, 
Slow off-diagonal decay for Szeg\"o kernels associated to smooth Hermitian line bundles, in: {\it  Harmonic analysis at Mount Holyoke (South Hadley, MA, 2001),} 
Contemp.\ Math. 320, Amer.\ Math.\ Soc., Providence, RI, 2003, pp.\ 77--89. 


\bibitem[De]{De} H. Delin, Pointwise estimates for the weighted Bergman
projection kernel in $\mathbf C\sp n$,
using a weighted $L\sp 2$ estimate for the $\overline\partial$
equation, {\it Ann.\ Inst.\ Fourier (Grenoble)} 48 (1998),
967--997.

\bibitem[FZ]{FZ} R. Feng and S. Zelditch, Median and mean of the supremum of $L^2$ normalized random holomorphic fields, {\it J.  Funct.\ Anal.}, to appear, arXiv:1303.4096.

\bibitem[Li]{Li} N. Lindholm,  Sampling in weighted $L^p$ spaces of entire
functions in $\C^n$ and estimates of the Bergman kernel, {\it J. Funct.\
Anal.} 182 (2001), 390--426.

\bibitem[MM]{MM} X. Ma and G. Marinescu, {\it
Holomorphic Morse inequalities and Bergman kernels\/},
Progress in Mathematics, 254, Birkh\"auser Verlag, Basel, 2007. 

\bibitem[SZ1]{SZ} B. Shiffman and S. Zelditch,  Asymptotics of almost
holomorphic sections of ample line bundles on symplectic
manifolds,  {\it J. Reine Angew.\ Math.}  544 (2002), 181--222.

\bibitem[SZ2]{SZlevy} B. Shiffman and S. Zelditch,  Random polynomials of high degree and Levy concentration of measure, {\it  Asian J. Math.}  7  (2003), 627--646.

\bibitem[SZ3]{SZvar} B. Shiffman and S.  Zelditch, Number variance
 of random zeros on complex manifolds, {\it Geom.\ Funct.\ Anal.} 18 (2008), 1422--1475.

\bibitem[SZZ]{SZZ}  B. Shiffman, S. Zelditch and S. Zrebiec, Overcrowding and hole
probabilities for random zeros on complex manifolds, {\it Indiana Univ.\ Math.\ J.} 57
(2008), 1977--1997.

\bibitem[SoZ]{SoZ} C. Sogge and S. Zelditch, Riemannian manifolds with maximal eigenfunction growth, {\it
Duke Math. J.}  114  (2002),  387--437.

\bibitem[Ti]{Ti} G.  Tian, On a set of polarized \kahler metrics on algebraic
manifolds, {\it J.  Diff.\ Geometry\/} 32 (1990), 99--130.

\bibitem[TZ]{TZ} J. Toth and S. Zelditch,
Riemannian manifolds with uniformly bounded eigenfunctions, {\it
Duke Math. J.} 111 (2002), 97Ð-132.

\bibitem[Va]{Va} J. VanderKam, L${}^\infty$ norms and quantum ergodicity on the sphere, {\it Internat.\ Math.\ Res.\ Notices} 1997 (1997), 329--347; Correction to ``L${}^\infty$ norms and quantum ergodicity on the sphere," {\it Internat.\ Math.\ Res.\ Notices}, 1998 (1998), 65. 

\bibitem[Ze]{Z} S. Zelditch, \szego kernels and a theorem of Tian,
{\it Internat.\ Math.\ Res.\ Notices} 1998 (1998),  317--331.


\end{thebibliography}
\end{document}